\documentclass[english]{amsart}

\usepackage{babel}
\usepackage{amsmath}
\usepackage{amsthm}
\usepackage{amssymb}
\usepackage[utf8]{inputenc}
\usepackage{graphicx}
\usepackage{float}
\usepackage[colorlinks = true, urlcolor = blue, linkcolor = blue, citecolor = blue, breaklinks = true]{hyperref}
\usepackage{breakcites}
\usepackage[dvipsnames]{xcolor}
\usepackage{tikz-cd}
\usetikzlibrary{spath3, intersections}
\usetikzlibrary{
  knots,
  hobby,
  decorations.pathreplacing,
  shapes.geometric,
  calc
}

\tikzset{
  knot diagram/every strand/.append style={
    ultra thick,
    red
  },
  show curve controls/.style={
    postaction=decorate,
    decoration={show path construction,
      curveto code={
        \draw [blue, dashed]
        (\tikzinputsegmentfirst) -- (\tikzinputsegmentsupporta)
        node [at end, draw, solid, red, inner sep=2pt]{};
        \draw [blue, dashed]
        (\tikzinputsegmentsupportb) -- (\tikzinputsegmentlast)
        node [at start, draw, solid, red, inner sep=2pt]{}
        node [at end, fill, blue, ellipse, inner sep=2pt]{}
        ;
      }
    }
  },
  show curve endpoints/.style={
    postaction=decorate,
    decoration={show path construction,
      curveto code={
        \node [fill, blue, ellipse, inner sep=2pt] at (\tikzinputsegmentlast) {}
        ;
      }
    }
  }
}
\usepackage{url}
\usepackage{dcolumn}
\usepackage{xcolor}
\usepackage{tikz}
\usetikzlibrary{decorations.markings}
\usepackage{mathtools}
\usepackage{nicematrix}
\usepackage{mathrsfs}
\usepackage{booktabs}
\usepackage{pdfpages}
\usepackage{bbm}
\usepackage{enumitem}
\usepackage{scalerel}
\usetikzlibrary{matrix,arrows,chains}
\newtheorem*{theorem*}{Theorem}
\newtheorem{prop}{Proposition}[section]
\newtheorem{thm}[prop]{Theorem}
\newtheorem*{thm*}{Theorem}
\newtheorem{cor}[prop]{Corollary}
\newtheorem{lem}[prop]{Lemma}
\newtheorem{conj}[prop]{Conjecture}
\theoremstyle{definition}
\newtheorem{rem}[prop]{Remark}
\newtheorem{de}[prop]{Definition}
\newtheorem{ex}[prop]{Example}

\newtheorem*{quest*}{Question}
\newtheorem*{cor*}{Corollary}
\newtheorem*{conj*}{Conjecture}

\usepackage[left=3cm,right=3cm]{geometry}
\usepackage{stackengine,wasysym}

\makeatletter
\tikzcdset{
  open/.code     = {\tikzcdset{hook, circled};},
  closed/.code   = {\tikzcdset{hook, slashed};},
  open'/.code    = {\tikzcdset{hook', circled};},
  closed'/.code  = {\tikzcdset{hook', slashed};},
  circled/.code  = {\tikzcdset{markwith = {\draw (0,0) circle (.375ex);}};},
  slashed/.code  = {\tikzcdset{markwith = {\draw[-] (-.0ex,-.4ex) -- (.0ex,.4ex);}};},
  markwith/.code ={
    \pgfutil@ifundefined%
    {tikz@library@decorations.markings@loaded}%
    {\pgfutil@packageerror{tikz-cd}{You need to say %
      \string\usetikzlibrary{decorations.markings} to use arrows with markings}{}}{}%
    \pgfkeysalso{/tikz/postaction = {
      /tikz/decorate,
      /tikz/decoration={markings, mark = at position 0.5 with {#1}}}
    }
  },
}
\makeatother

\DeclareMathOperator{\Z}{\mathbb{Z}}

\DeclareMathOperator{\FinSet}{FinSet}

\DeclareMathOperator{\id}{id}

\DeclareMathOperator{\cl}{cl}

\DeclareMathOperator{\sgn}{sgn}

\DeclareMathOperator{\ab}{ab}

\DeclareMathOperator{\Fun}{Fun}

\DeclareMathOperator{\Aut}{Aut}

\DeclareMathOperator{\im}{im}

\renewcommand{\ker}{\kerr}
\DeclareMathOperator{\kerr}{ker}
\DeclareMathOperator{\coker}{coker}
\renewcommand{\lim}{\limi}
\DeclareMathOperator{\limi}{lim}

\DeclareMathOperator{\copro}{\scaleobj{0.8}{\amalg}}

\DeclareMathOperator{\Mat}{Mat}

\DeclareMathOperator{\rk}{rk}

\DeclareMathOperator{\HH}{H}

\newcommand{\vdotsc}{%
  \vcenter{\hbox{\rotatebox{90}{$\dotsc$}}}%
}

\newcommand{\ddotsc}{%
  \vcenter{\hbox{\rotatebox{-45}{$\dotsc$}}}%
}

\graphicspath{ {./Bilder/} }
\tikzset{double line with arrow/.style args={#1,#2}{decorate,decoration={markings,%
mark=at position 0 with {\coordinate (ta-base-1) at (0,1pt);
\coordinate (ta-base-2) at (0,-1pt);},
mark=at position 1 with {\draw[#1] (ta-base-1) -- (0,1pt);
\draw[#2] (ta-base-2) -- (0,-1pt);
}}}}

\tikzset{%
    symbol/.style={%
        draw=none,
        every to/.append style={%
            edge node={node [sloped, allow upside down, auto=false]{$#1$}}}
    }
}

\tikzset{
  open/.style = {decoration = {markings, mark = at position 0.5 with { \node[transform shape, scale = .8] {$\circ$}; } }, postaction = {decorate} }
}

\begin{document}

\title{$K_1$ of Matroids}

\author[Wittich]{Thor Wittich}
\address[Wittich]{
Department of Mathematics, Universit\"{a}t Osnabr\"{u}ck, 
Albrechtstrasse 28a, 49076 Osnabr\"{u}ck,
Germany}
\email{thor.wittich@uni-osnabrueck.de}

\begin{abstract}
Using the geometric realization spectral sequence and the Serre spectral sequence, we compute $K_1(\Mat_*)$, the first K-theory group of the proto-exact category of pointed matroids, answering a question of Jun, Sistko and Wright.
\end{abstract}
\maketitle

{ \hypersetup{hidelinks} \tableofcontents }

\section{Introduction}
Matroids are combinatorial objects which abstract and encode the notion of independence, including linear independence within vector spaces and the absence of cycles in a finite graph. As such, they play a prominent role in various fields ranging from combinatorics and discrete mathematics to optimization and algebra. Also in the field of tropical geometry, which studies algebraic varieties from a combinatorial/linearized perspective, matroids are ubiquitous. They have in particular been used by Engel, de Gaay Fortman and Schreieder in \cite{EGS25} to obtain interesting results in algebraic geometry, such as the stable irrationality of very general cubic threefolds and the failure of the integral Hodge conjecture for very general principally polarized abelian varieties of dimension at least $4$.

\hfill

Algebraic $K$-theory is a powerful tool from homotopy theory which, in its simplest description, studies decompositions of objects. In recent years, the application of this tool to combinatorial objects has received increasing attention. The starting point of this was Zakharevich’s work on the scissors congruence $K$-theory of assemblers \cite{zbMATH06043101} and \cite{zbMATH06642280}. Her constructions associate $K$-theory spectra and hence $K$-theory groups to suitable categories equipped with scissors congruence/cutting and pasting operations, recovering, for instance, classical scissors congruence groups and the Grothendieck ring of varieties from algebraic geometry. The higher $K$-groups provide further invariants of decomposition problems and are an active field of study. For example, recent work of Kupers, Lemann, Malkiewich, Miller and Sroka \cite{arXiv:2408.08081} relates scissors congruence $K$-theory to the homology of scissors automorphism groups, providing a group-theoretic interpretation of these higher invariants. The authors describe $K_1$ as the abelianization of scissors automorphism groups and as such further illustrate how combinatorial $K$-theory connects decomposition relations with the symmetries which preserve them.

\hfill

The $K$-theory of the category of pointed matroids $\Mat_*$, as first considered in \cite{zbMATH07242437}, is part of this development. The relevant decompositions stem from restrictions $\!\left.M\right|\!A$ of a matroid $M$ together with their complementary contractions $M/A$. In \cite{zbMATH07242437}, the authors show that these notions equip the category $\Mat_*$ with the structure of a proto-exact category, which is a not necessarily additive variant of exact categories. By \cite{zbMATH07103772}, this is one of the frameworks which enable $K$-theory. The resulting algebraic $K$-theory $K(\Mat_*)$ of pointed matroids hence studies and encodes restriction-contraction decompositions of matroids.

\hfill

Currently, not much about the homotopy groups of $K(\Mat_*)$, that is, about the $K$-groups of pointed matroids, is known. Eppolito, Jun and Szczesny showed in \cite{zbMATH07242437} that $K_0(\Mat_*) \cong \Z^2$ and that the stable homotopy groups of spheres embed into the $K$-groups: $\pi_n(\mathbb{S}) \hookrightarrow K_n(\Mat_*)$. In particular, the $K$-groups contain interesting information of a homotopical nature, but how much more information do the $K$-groups contain? This is exactly what Jun, Sistko and Wright ask in Remark 5.6 of \cite{zbMATH08208138}. Our main result answers this question in degree one:
\begin{theorem*}[Theorem \ref{main result}]
There is an isomorphism $K_1(\Mat_*) \cong \mathbb{Z}/2\mathbb{Z} \cong \pi_1(\mathbb{S})$.
\end{theorem*}
In particular, the only information comes from the first stable homotopy group of spheres.

\hfill

Our main tool is the geometric realization spectral sequence for homology. It expresses the relevant homology groups in terms of automorphisms of matroids and their flags, and we use two key ideas to get a hold of these groups. The first one is that we can use so-called free coextensions of matroids to raise the rank of a matroid without changing its corresponding automorphism class in the relevant term of the spectral sequence. This allows us to assume that all pointed matroids are free and hence that they are basically nothing but finite pointed sets. The second one is to use suitable successive decompositions into matroids with one element and then to study these decompositions into simple pieces to find generators. This allows us to compute the second homology of the delooped K-theory space of pointed matroids. To put everything together, we use the Serre spectral sequence, which allows us to relate this homology group to the desired homotopy group, resulting in the above computation. Based on our arguments, we give the following conjectural description of the entire $K$-theory spectrum:
\begin{conj*}[Conjecture \ref{conj}]
The map $\operatorname{K}(\Mat_*) \to \HH\!\Z \vee \, \mathbb{S}$ induced by the rank map and by the forgetful map to pointed finite sets is an equivalence of spectra. In particular, we have isomorphisms $K_n(\Mat_*) \cong \pi_n(\mathbb{S})$ for all positive $n$.
\end{conj*}
This article is structured as follows. Section 2 introduces and collects relevant facts about matroids and their pointed versions. Section 3 briefly reviews proto-exact categories and their K-theory, including the aforementioned description in terms of flags of matroids. Finally, Section 4 is dedicated to the computations with the geometric realization and Serre spectral sequences and concludes with the main result and the  conjectural description from above.

\hfill

\textbf{Acknowledgements.} The author wishes to thank Ming Ng for making him aware of the problem of computing the $K$-theory of matroids, specifically in degree $1$. Furthermore, the author used ChatGPT 6-Astra for proofreading and to search for references. Except for pointing out typos and the like, it found combinatorial slipups in the proof of Proposition \ref{class of flag in terms of z} and fixed them.

\section{The Category of Matroids}
In this section we give a quick introduction to matroids and generally refer the reader to \cite{zbMATH05873618} or \cite{zbMATH03534506} for more information. Despite this, we will provide arguments for a few standard lemmas to get more homotopically minded readers used to the combinatorial notions. 

\hfill

Matroids are combinatorial objects which abstract the notion of linear independence from vector spaces. There are various equivalent definitions, of which we will use and present three. As in \cite{zbMATH07242437}, the one in terms of flats will be the most relevant one for this article.
\begin{de}
A matroid $M$ is a pair $(E,\mathcal{I})$ consisting of a finite set $E(M)$, called the ground set of the matroid $M$, together with a collection $\mathcal{I}(M) \subset \mathcal{P}(E(M))$ of independent subsets subject to:
\begin{enumerate}
\item[(i)] The empty set is independent, i.e., we have $\emptyset \in \mathcal{I}(M)$.
\item[(ii)] The set $\mathcal{I}$ is hereditary, i.e., if $A \in \mathcal{I}(M)$ and $A' \subset A$, then $A' \in \mathcal{I}(M)$.
\item[(iii)] The set $\mathcal{I}(M)$ satisfies the Steinitz exchange property, that is, whenever $A, B \in \mathcal{I}(M)$ with $\lvert A \rvert < \lvert B \rvert$, then there exists an element $e \in B \setminus A$ such that $A \cup \{ e \} \in \mathcal{I}(M)$.
\end{enumerate}
\end{de}
As in linear algebra, any maximally independent set is called a basis of the matroid $M$ and the size of any two bases agrees; see Lemma 1.2.1 of \cite{zbMATH05873618}. This number is called the rank of the matroid $M$ and is denoted by $\rk(M)$. More generally, the rank $\rk_M(A)$ of any subset $A$ of a matroid $M$ is defined as the size of the largest independent set contained in $A$. 

\hfill

Before we get to some central examples, let us introduce two further notions.
\begin{de}
Let $E$ be a finite set. A closure operator on $E$ is a map $\cl \colon \mathcal{P}(E) \to \mathcal{P}(E)$ subject to:
\begin{enumerate}
\item[(i)] The map $\cl$ is extensive, i.e., for all subsets $A \subset E$, we have $A \subset \cl(A)$.
\item[(ii)] The map $\cl$ is idempotent, i.e., we have $\cl(\cl(A)) = \cl(A)$ for all subsets $A \subset E$.
\item[(iii)] The map $\cl$ is monotone, i.e., for all subsets $A,B \subset E$ with $A \subset B$ we have $\cl(A) \subset \cl(B)$.
\end{enumerate}
If additionally for all subsets $A \subset E$ and all $e \in E$, whenever $f \in \cl(A \cup \{ e \}) \setminus \cl(A)$, we have $e \in \cl(A \cup \{ f \}) \setminus \cl(A)$, we say that the closure operator $\cl$ satisfies the Steinitz exchange property (SEP).
\end{de}
To state the definition of flats, let us quickly clarify the notion of cover used here. Given two subsets $A$ and $B$ of some finite set $E$, we say that $B$ covers $A$, if $A \subsetneq B$ and there is no subset $C \subset E$ with $A \subsetneq C \subsetneq B$. In other words, $B$ covers $A$ if $B$ is minimal among subsets of $E$ which properly contain $A$.
\begin{de}
Let $E$ be a finite set. A collection $\mathcal{F} \subset \mathcal{P}(E)$ is a set of flats on $E$, if it satisfies the following axioms:
\begin{enumerate}
\item[(i)] The set $E$ is a flat, i.e., $E$ is contained in $\mathcal{F}$.
\item[(ii)] The set of flats $\mathcal{F}$ is stable under finite intersections, i.e., we have $F \cap F' \in \mathcal{F}$ for all $F,F' \in \mathcal{F}$.
\item[(iii)] For all flats $F \in \mathcal{F}$, the set $\{ F' \setminus F \mid F' \text{ covers } F \text{ within } \mathcal{F} \}$ is a partition of $E \setminus F$.
\end{enumerate}
\end{de}
A basic result of matroid theory now identifies all those notions:
\begin{prop}\label{matroid characterizations}
Given a finite set $E$, the following data are equivalent:
\begin{enumerate}
\item[(i)] A matroid $M$ with ground set $E$;
\item[(ii)] A closure operator $\cl$ on $E$ satisfying the Steinitz exchange property;
\item[(iii)] A set of flats on $E$.
\end{enumerate}
\end{prop}
While we do not want to present a proof of this standard fact, for which we refer the reader to Corollary 1.4.7 and Exercise 11 in Section 1.4 in \cite{zbMATH05873618}, we do want to indicate how to pass between these notions:
\vspace{4pt}
\begin{center}
\begin{tikzcd}
\{ \text{Matroids on } E \} \arrow[r, bend left, "\mathcal{I} \, \mapsto \, (A \,\mapsto \,\{ x \in E \, \mid \, \rk_M(A \cup \{ x \}) \, = \, \rk_M(A) \})"] & \{ \text{Closure operators on } E \text{ satisfying the SEP} \} \arrow[l, bend left, "\cl \, \mapsto \, \{ A \subset E \, \mid \, x \notin \cl(A \setminus \{ x \}) \text{ for all } x \in A \}"] \arrow[r, bend left, "\cl \, \mapsto \,\im(\cl)"] & \{ \text{Sets of flats on } E \} \arrow[l, bend left, "\mathcal{F} \, \mapsto \, (A \, \mapsto \, \bigcap_{A \subset F \in \mathcal{F}} F)"]
\end{tikzcd}
\end{center}
While this is relevant in matroid theory, it will actually not be of importance for this article to know how to pass between these equivalent characterizations of matroids. We only need to know that we can verify that two matroids agree by choosing one of the three notions, which gives us quite some flexibility.

\hfill

Let us finally present a few examples, some of which we will make use of later on.
\begin{ex}
$\phantom{1}$
\begin{enumerate}
\item[(i)] A vectorial matroid $M$ has a finite subset $E$ of a vector space $V$ as ground set and the independent sets are given by linearly independent sets within $E$. The closure operator maps a subset $A \subset E$ to the set of all vectors in $E$ lying in the linear span of $A$, and the flats of a vectorial matroid $M$ are those subsets $A \subset E$ with $A = \langle A \rangle \cap E$.
\item[(ii)] A graphic matroid $M$ has the edges $E$ of a finite graph $\Gamma = (V,E)$ as ground set and the independent sets are the forests, i.e., those collections of edges such that there is at most one path between any two of their vertices. The closure operator maps a subset $A \subset E$ to the set of all edges connecting vertices already connected by paths in $A$, and the flats of a graphic matroid are those subsets $A \subset E$ that contain all paths staying entirely within the connected component's set of vertices.
\item[(iii)] Given non-negative integers $r$ and $n$ with $r \leq n$, the uniform matroid $U_{r,n}$ on a finite set $E$ with $n$ elements has every subset $A \subset E$ of cardinality at most $r$ as an independent subset. The closure operator maps subsets of cardinality smaller than $r$ to themselves and subsets of cardinality at least $r$ to $E$, and the flats are the subsets of cardinality smaller than $r$ and $E$ itself. Note that it is customary to denote this matroid by $U_{r,n}$ independently of the ground set.
\item[(iv)] As a notable special case of the previous example we have the free matroid on $E$, which is the uniform matroid $U_{n,n}$. Here every subset is independent, the closure operator is just the identity map, and every subset is a flat.
\end{enumerate}
\end{ex}
Now we come to the notion of morphisms between matroids. While there are multiple meaningful choices, the most common one is the following.
\begin{de}
A map $\varphi \colon M \to M'$ between two matroids is strong if $\varphi^{-1}(F')$ is a flat of $M$ for every flat $F' \in \mathcal{F}(M')$.
\end{de}
\begin{ex}\label{ex strong maps}
In the setups from linear algebra and graph theory, strong maps arise as expected:
\begin{enumerate}
\item[(i)] If $M$ and $M'$ are two vectorial matroids with respect to subsets $E$ and $E'$ from two vector spaces $V$ and $V'$ over a field $F$, then any linear map $f \colon V \to V'$ satisfying $f(E) \subset E'$ yields a strong map $\varphi \colon M \to M'$ by restriction.
\item[(ii)] If $M$ and $M'$ are two graphic matroids associated with graphs $\Gamma$ and $\Gamma'$, then any graph morphism $\Gamma \to \Gamma'$ yields a strong map $M \to M'$.
\end{enumerate}
\end{ex}
Matroids together with strong maps form a category $\Mat$, as is not difficult to verify. Its properties are for instance studied in the article \cite{zbMATH06859742}. Let us note here that a bijective strong map $\varphi \colon M \to M'$ between matroids of the same rank is already an isomorphism, which follows from Corollary 7.3.4 together with Proposition 7.3.6 of \cite{zbMATH05873618}. In particular, a bijective strong self-map is an automorphism.

\hfill

There are various operations on matroids which are crucial for this article, some of which yield the decompositions of matroids which the $K$-theory studies. Let us introduce those:
\begin{de}
Let $M$ be a matroid with ground set $E$ and let $A$ be a subset of $E$.
\begin{enumerate}
\item[(i)] The restriction $\left.M\right|\!A$ of $M$ to $A$ is the matroid with ground set $A$ and
\begin{enumerate}
\item[-] independent sets $\mathcal{I}(\!\left.M\right|\!A) = \{ I \in \mathcal{I}(M) \mid I \subset A \}$
\item[-] closure operator $\cl_{\!\left.M\right|A}(S) = \cl_M(S) \cap A$
\item[-] flats $\mathcal{F}(\!\left.M\right|\!A) = \{ F \cap A \mid F \in \mathcal{F}(M) \}$
\end{enumerate}
\item[(ii)] The contraction $M / A$ of $M$ with respect to $A$ is the matroid with ground set $E \setminus A$ and
\begin{enumerate}
\item[-] independent sets $\mathcal{I}(M / A) = \{ I \subset E \setminus A \mid I \cup J \in \mathcal{I}(M) \}$ for a fixed basis $J$ of $\left.M\right|\!A$
\item[-] closure operator $\cl_{M / A}(S) = \cl_M(S \cup A) \setminus A$
\item[-] flats $\mathcal{F}(M / A) = \{ F \setminus A \mid F \in \mathcal{F}(M) \text{ with } A \subset F \}$
\end{enumerate}
\end{enumerate}
\end{de}
The restriction to the complement of $S$ is also called the deletion of $S$ and denoted by $M \setminus S$.

These notions have some compatibility, as the following simple lemma shows, which can be found as Proposition 3.1.25 (iii) in \cite{zbMATH05873618}.
\begin{lem}\label{contraction and restrictrion commute}
Let $M$ be a matroid and let $A$ and $B$ be subsets of $E(M)$ with $A \subset B$. Then we have $\left.(M/A)\right|\!(B \setminus A) = (\!\left.M\right|\!B)/A$.
\end{lem}
\begin{proof}
First note that both of these matroids have ground set $B \setminus A$. Now choose a basis $I$ of $\left.M\right|\!A$. This is then also a basis of $\left.\left.(M\right|\!B)\right|\!A$. Now we have by definition that a subset $J \subset B \setminus A$ is independent in $\left.(M/A)\right|\!(B \setminus A)$ if and only if $J \cup I$ is independent in $M$, which is equivalent to $J \cup I$ being independent in $\left.M\right|\!B$ as $J \cup I \subset B$. Moreover, this last condition is equivalent to $J$ being independent in $(\!\left.M\right|\!B)/A$, which finishes the proof.
\end{proof}
The category of matroids admits finite coproducts, defined as follows.
\begin{de}
The direct sum $M \oplus M'$ of two matroids $M$ and $M'$ is the matroid with ground set $E(M) \copro \, E(M')$ and
\begin{enumerate}
\item[-] independent sets $\mathcal{I}(M \oplus M') = \{ I \copro I' \mid I \in \mathcal{I}(M) \text{ and } I' \in \mathcal{I}(M') \}$
\item[-] closure operator $\cl_{M \oplus M'}(A) = \cl_M(A \cap E(M)) \copro \, \cl_{M'}(A \cap E(M'))$
\item[-] flats $\mathcal{F}(M \oplus M') = \{ F \copro F' \mid F \in \mathcal{F}(M) \text{ and } F' \in \mathcal{F}(M') \}$
\end{enumerate}
\end{de}
A central operation for us will be the free coextension of a matroid by a single element.
\begin{de}
Let $M$ be a matroid with ground set $E$ and let $t \notin E$. The free coextension $M \times t$ of $M$ by $t$ is the matroid with ground set $E \cup \{ t \}$ and
\begin{enumerate}
\item[-] independent sets $\mathcal{I}(M \times t) = \{ I \cup J \mid I \in \mathcal{I}(M) \text{ and } J \subset E \cup \{ t \} \text{  with } \lvert J \rvert \leq 1 \}$
\item[-] closure operator $\cl_{M \times t}$ mapping a subset $A \subset E \cup \{ t \}$ to $A$ if $A$ is contained in $E$ and is independent, and mapping it to $\cl_M(A \setminus \{ t \}) \cup \{ t \}$ otherwise
\item[-] flats $\mathcal{F}(M \times t) = \mathcal{I}(M) \cup \{ F \cup \{ t\} \mid F \in \mathcal{F}(M) \}$
\end{enumerate}
\end{de}
Based on the above description of the independent sets we immediately obtain $\rk(M \times t) = \rk(M) + 1$. Furthermore, the above definitions are made to satisfy the following.
\begin{lem}\label{basic contr comp}
Let $M$ be a matroid with ground set $E$ and let $t \notin E$. Then we have $(M \times t) / \{ t \} = M$ and $\left.(M \times  t )\right|\! \{ t \}  = U_{1,1}$. If $M$ is non-free, we furthermore have $(M \times t) / E = U_{0,1}$ and if $M$ is free, then $(M \times t) / E = U_{1,1}$.
\end{lem}
\begin{proof}
We verify the first claim in terms of the closure operator. Here we have
$$\cl_{(M \times t) / \{ t \}}(A) = \cl_{M \times t}(A \cup \{ t \}) \setminus \{ t\} = (\cl_M(A) \cup \{ t \}) \setminus \{ t \} = \cl_M(A)$$
for all subsets $A \subset E$, so that $(M \times t) / \{ t \} = M$.

For the second claim, note that since $\emptyset \in \mathcal{I}(M)$, we have $\{ t \} \in \mathcal{I}(M \times t)$ and thus $\{ t \} \in \mathcal{I}(\!\left.(M \times  t )\right|\! \{ t \})$. In particular, $\left.(M \times  t)\right|\! \{ t \} = U_{1,1}$ is free on $\{ t \}$.

Now suppose that $M$ is not free, so that $E$ is a dependent set of $M$. Then
$$\cl_{(M \times t) / E}(\emptyset) = \cl_{M \times t}(E) \setminus E = (\cl_M(E) \cup \{ t \}) \setminus E = (E \cup \{ t \}) \setminus E = \{  t \}$$
and we obtain $(M \times t) / E = U_{0,1}$. If $M$ is free instead, then $E$ is an independent set of $M$ and we get
$$\cl_{(M \times t) / E}(\emptyset) = \cl_{M \times t}(E) \setminus E = E \setminus E = \emptyset,$$
so that $(M \times t) / E = U_{1,1}$. This finishes the proof.
\end{proof}
Moreover, automorphisms of matroids can be extended to free coextensions, which will become very useful in our $K$-theory computations later on. This statement follows from Proposition 7.2.6 and Proposition 7.3.9 of \cite{zbMATH05873618}, but their proofs use notions we have not introduced here. Therefore, we will give a direct proof for the convenience of the reader.
\begin{lem}\label{auto ext}
Let $M$ be a matroid and let $M \times t$ be the free coextension of $M$ by $t$. If $\alpha \colon M \to M$ is an automorphism, then so is its extension $\beta \colon M \times t \to M \times t$ defined by mapping $t$ to $t$.
\end{lem}
\begin{proof}
Since the map $\beta$ is clearly bijective, it suffices to verify that it is strong, as we observed just below Example \ref{ex strong maps}. By definition, a flat of $M \times t$ is either an independent set $I \in \mathcal{I}(M)$ or $F \cup \{ t \}$ for some flat $F \in \mathcal{F}(M)$. In the first case the preimage $\beta^{-1}(I) = \alpha^{-1}(I)$ is independent and hence a flat of $M \times t$, and in the second case $\beta^{-1}(F \cup \{ t \}) = \alpha^{-1}(F) \cup \{ t \}$ is also a flat of $M \times t$. So $\beta$ is strong and thus an automorphism.
\end{proof}
Let us note that the above statement is false for endomorphisms, as can be easily seen from constant maps between uniform matroids.

\hfill

To define pointed matroids, we call an element $e \in E(M)$ a loop if it is not contained in any independent set of the matroid $M$. Equivalently, $e$ is a loop if it is contained in the unique minimal flat of $M$. As such, loops play the role of the zero vector in vector spaces.
\begin{de}
A pointed matroid is a pair $(M,*_M)$ where $M$ is a matroid and $*_M \in E(M)$ is a loop of $M$.
\end{de}
As is done in topology, we usually omit the base point from the notation. However, we will need to distinguish the ground set $E(M)$ and its reduced variant $\widetilde{E}(M) = E(M) \setminus \{ *_M \}$.
\begin{ex}
In the pointed setting, we add a distinguished loop to uniform matroids. In particular, $\lvert \widetilde{E}(U_{r,n}) \rvert = n$ and $\lvert E(U_{r,n}) \rvert = n + 1$. This makes $U_{0,0}$ the zero matroid, and free pointed matroids are those matroids that are free after removing the base point; see also Section 4.2 of \cite{zbMATH06859742}.
\end{ex}

\begin{de}
A map $\varphi \colon M \to M'$ between two pointed matroids is strong if $\varphi(*_M) = *_{M'}$ and $\varphi^{-1}(F')$ is a flat of $M$ for every flat $F' \in \mathcal{F}(M')$.
\end{de}
This gives rise to a category $\Mat_*$, the category of pointed matroids, which is studied in more detail in \cite{zbMATH06859742}. Here it is also verified that this category admits a forgetful functor $\Mat_* \to \FinSet_*$ and has finite coproducts called direct sums, which we recall as part of the following list, which mostly stems from Remark 2.9 of \cite{zbMATH07242437}:
\begin{rem}
For a pointed matroid $M$ and a subset $A \subset \widetilde{E}(M)$ we adopt the following conventions regarding the various operations that we have introduced in the unpointed setting:
\begin{enumerate}
\item[-] The restriction $\left.M\right|\!A$ is the unpointed restriction $\!\left.M\right|\! (A \cup \{ *_M \})$ including the base point. In particular, the restriction of an automorphism acts (trivially) on the base point.
\item[-] The contraction $M/A$ is the usual unpointed contraction $M/A$. The base point is not allowed to be contracted.
\item[-] The direct sum $M \oplus M'$ of $M$ with another pointed matroid $M'$ is the usual unpointed direct sum $M \oplus M'$ with the two base points identified. In other words, the ground set is $\widetilde{E}(M) \copro \, \widetilde{E}(M') \copro \, \{ * \}$ with base point $*$, and the independent sets and flats are disjoint unions of independent sets and flats of the two matroids, with the base points identified for the flats.
\item[-] The free coextension $M \times t$ of $M$ by an element $t$ is the usual unpointed free coextension with the base point added as a loop $((M \setminus \{ *_M \}) \times t) \copro \, \{ *_M \}$.  Here the deletion is in the unpointed setting.
\end{enumerate}
\end{rem}
With those conventions in place, Lemma \ref{basic contr comp} and Lemma \ref{auto ext} also hold in the pointed setting with $E$ replaced by $\widetilde{E}$, as can be quickly verified. Therefore, we will also refer to those statements when we use them in the pointed setting.
\begin{ex}
Let us conclude this section with two central examples of pointed strong maps, also found as Example 2.11.3 of \cite{zbMATH07242437}:
\begin{enumerate}
\item[(i)] Inclusion maps of the form $\left.M\right|\!A \to M$ are by definition pointed strong maps.
\item[(ii)] Quotient maps of the form $M \to M/A$ are by definition pointed strong maps. Here all the elements contained in $A$ are mapped to the base point.
\end{enumerate}
\end{ex}
These types of maps rise to the exact sequences encoded in the categorical structure of the next section.

\section{Proto-exact Categories and their K-theory}
Following Dyckerhoff and Kapranov \cite{zbMATH07103772}, we give a brief recollection of proto-exact categories, which are not necessarily additive variants of exact categories, and of the $S_{\bullet}$-construction of K-theory for such categories. This yields the algebraic K-theory space $K(\Mat_*)$ of pointed matroids.

\begin{de}
A proto-exact category is a category $\mathcal{C}$ together with two collections of morphisms $\mathcal{M}$ and $\mathcal{E}$, called admissible monomorphisms and admissible epimorphisms, respectively, subject to:
\begin{enumerate}
\item[(i)] The category $\mathcal{C}$ has a zero object $0$, and any morphism of the form $0 \to X$ lies in $\mathcal{M}$ and any morphism of the form $X \to 0$ lies in $\mathcal{E}$.
\item[(ii)] The collections $\mathcal{M}$ and $\mathcal{E}$ contain all isomorphisms and are closed under composition.
\item[(iii)] A commutative square 
\begin{center}
\begin{tikzcd}
X \arrow[r] \arrow[d] & Y \arrow[d] \\
Z \arrow[r] & W
\end{tikzcd}
\end{center}
with horizontal morphisms from $\mathcal{M}$ and vertical morphisms from $\mathcal{E}$ is cartesian if and only if it is cocartesian.
\item[(iv)] Any diagram in $\mathcal{C}$ of the form 
\begin{center}
\begin{tikzcd}
X \arrow[r] \arrow[d] & Y \\
Z & 
\end{tikzcd}
\hspace{5pt}or\hspace{5pt}
\begin{tikzcd}
 & Y \arrow[d] \\
Z \arrow[r] & W
\end{tikzcd}
\end{center}
with horizontal morphisms from $\mathcal{M}$ and vertical morphisms from $\mathcal{E}$ can be completed to a bicartesian square as in the previous point.
\end{enumerate}
\end{de}
Admissible monomorphisms are typically denoted by $\hookrightarrow$ and admissible epimorphisms by $\twoheadrightarrow$, despite the fact that those symbols could also stand for monomorphisms and epimorphisms in the category $\mathcal{C}$. A bicartesian square of the form
\begin{center}
\begin{tikzcd}
X \arrow[r, hookrightarrow] \arrow[d, twoheadrightarrow] & Y \arrow[d, twoheadrightarrow] \\
0 \arrow[r, hookrightarrow] & W
\end{tikzcd}
\end{center}
is usually denoted by $X \hookrightarrow Y \twoheadrightarrow W$ and is called an admissible exact sequence of $\mathcal{C}$. 

\hfill

A natural question is, of course, how exactly proto-exact categories relate to their classical counterpart. 
\begin{rem}\label{exact vs proto}
For a proto-exact category to be exact, we would need to require $\mathcal{C}$ to be additive instead of just being pointed, $\mathcal{M}$ to be closed under all pushouts and not only under pushouts along $\mathcal{E}$, and $\mathcal{E}$ to be closed under all pullbacks and not only under pullbacks along $\mathcal{M}$. For more details on this comparison, we refer the reader to Section 1 of \cite{arXiv:2503.05940}.
\end{rem}
\begin{ex}
There are two central examples of interest to us here:
\begin{enumerate}
\item[(i)] The category $\FinSet_*$ of pointed finite sets is a proto-exact category. Here the admissible monomorphisms are the injections and the admissible epimorphisms are those surjections whose fibers over non-base points are singletons; see Example 2.4.4 (a) of \cite{zbMATH07103772}. 
\item[(ii)] The category $\Mat_*$ of pointed matroids is a proto-exact category, as shown in Section 5 of \cite{zbMATH07242437}. Here $\mathcal{M}$ consists of inclusions of the form $\left.M\right|\!A \hookrightarrow M$ up to isomorphism and $\mathcal{E}$ consists of quotient maps $M \twoheadrightarrow M/A$ up to isomorphism, where in both cases $A \subset \widetilde{E}(M)$.
\end{enumerate}
\end{ex}
As explained in Remark \ref{exact vs proto}, exact categories are examples of proto-exact with the same exact structure. For examples from quadratic form theory or Arakelov geometry, we refer the reader to Section 2.4 of \cite{zbMATH07103772}, and for examples related to $\mathbb{F}_1$-geometry, we refer the reader to \cite{zbMATH07525043}.

\hfill

Let us now turn our focus towards K-theory. Originally, Waldhausen's $S_{\bullet}$-construction applies to categories with weak equivalences and cofibrations, nowadays usually referred to as Waldhausen categories \cite{zbMATH03927168}. As explained in Section 2.4 of \cite{zbMATH07103772}, it is also available for proto-exact categories. Let us recall the construction.

\hfill

For a non-negative integer $n$, we consider the functor category $\mathcal{T}_n = \Fun([1],[n])$ which corresponds to the poset of ordered pairs $0 \leq i \leq j \leq n$ with the order $(i,j) \leq (i',j')$ if $i \leq i'$ and $j \leq j'$. Then defining a functor $F \colon \mathcal{T}_n \to \mathcal{C}$ amounts to a choice of a commutative diagram in $\mathcal{C}$ of the form
\begin{center}
\begin{tikzcd}
F(0,0) \arrow[r] & F(0,1) \arrow[r] \arrow[d] & \dotsc \arrow[r] & F(0,n-1) \arrow[r] \arrow[d] & F(0,n) \arrow[d] \\
& F(1,1) \arrow[r] & \dotsc \arrow[r] & F(1,n-1) \arrow[r] \arrow[d] & F(1,n) \arrow[d] \\
& & \ddotsc & \vdotsc \arrow[d] & \vdotsc \arrow[d] \\
& & & F(n-1,n-1) \arrow[r] & F(n-1,n) \arrow[d] \\
& & & & F(n,n) \\
\end{tikzcd}
\end{center}
If now $\mathcal{C}$ is a proto-exact category, we let $S_n(\mathcal{C})$ denote the maximal subgroupoid of the full subcategory of those functors $F \in \Fun(\mathcal{T}_n, \mathcal{C})$ satisfying
\begin{enumerate}
\item[(i)] $F(i,i) = 0$ for all $0 \leq i \leq n$;
\item[(ii)] The horizontal arrows are admissible monomorphisms and the vertical arrows are admissible epimorphisms;
\item[(iii)] All squares are bicartesian.
\end{enumerate}
This construction comes with face and degeneracy maps by omitting the $k$-th row and $k$-th column or inserting an additional row and column at the $k$-th position via identity maps. It is then not difficult to verify that one obtains a simplicial groupoid $S_{\bullet}(\mathcal{C})$.

\hfill

Dyckerhoff and Kapranov provide another description of the groupoids $S_n(\mathcal{C})$, which we make use of. For this, we let $\mathcal{M}_n$ denote the groupoid of chains of $n - 1$ composable admissible monomorphisms 
\begin{center}
\begin{tikzcd}
X_1 \arrow[r, hookrightarrow] & \dotsc \arrow[r, hookrightarrow] & X_n 
\end{tikzcd}
\end{center}
and their isomorphisms. We refer to this as the groupoid of flags of length $n - 1$. Note that this makes $\mathcal{M}_0$ the terminal groupoid.
\begin{lem}\label{Flag groupoid description}
Let $\mathcal{C}$ be a proto-exact category. For all positive integers $n$, the functor $\mu_n \colon S_n(\mathcal{C}) \to \mathcal{M}_n$ mapping a diagram $F$ to the subdiagram
\begin{center}
\begin{tikzcd}
F(0,1) \arrow[r, hookrightarrow] & \dotsc \arrow[r, hookrightarrow] & F(0,n) 
\end{tikzcd}
\end{center}
is an equivalence of groupoids. Under these equivalences, the face maps $d_i \colon S_n(\mathcal{C}) \to S_{n-1}(\mathcal{C})$ are identified with
\begin{enumerate}
\item[(i)] the map that omits the $i$-th entry for all $1 \leq i \leq n$;
\item[(ii)] the map that sends a flag
\begin{center}
\begin{tikzcd}
X_1 \arrow[r, hookrightarrow] & \dotsc \arrow[r, hookrightarrow] & X_n 
\end{tikzcd}
\end{center}
of length $n-1$ to the flag of quotients
\begin{center}
\begin{tikzcd}
X_2/X_1 \arrow[r, hookrightarrow] & \dotsc \arrow[r, hookrightarrow] & X_n/X_1 
\end{tikzcd}
\end{center}
for $i = 0$.
\end{enumerate}
\end{lem}
\begin{proof}
The equivalence has been established in Lemma 2.4.9 (i) of \cite{zbMATH07103772}. For $1 \leq i \leq n$, the description of the face maps is clear from the definition of $\mu_n$. Thus it only remains to understand the case of $i = 0$. Applying $d_0$ to a diagram $F \in S_n(\mathcal{C})$ leaves
\begin{center}
\begin{tikzcd}
F(1,2) \arrow[r, hookrightarrow] & \dotsc \arrow[r, hookrightarrow] & F(1,n) 
\end{tikzcd}
\end{center}
as the new first row. Now we want to identify this row in terms of the old first row. Since all the squares in $F$ are bicartesian, we have pushout squares of the form
\begin{center}
\begin{tikzcd}
F(0,1) \arrow[r] \arrow[d] & F(0,j) \arrow[d] \\
0 \arrow[r] & F(1,j)
\end{tikzcd}
\end{center}
for all $2 \leq j \leq n$ by composing the squares horizontally. In particular, we obtain isomorphisms $F(1,j) \cong F(0,j)/F(0,1)$, yielding the claimed description of the face maps also for $i = 0$.
\end{proof}
In the setting of Waldhausen categories, this lemma is also known, as the reader might expect. It can be found as Definition IV.8.3.1 in Weibel's $K$-book \cite{zbMATH06123945} together with the paragraph immediately following it. As it is quite crucial for Section $4$, we still wanted to include the short argument dewcribing the face maps.

\hfill

In the case of matroids, we will often abbreviate a flag
\begin{center}
\begin{tikzcd}
\left.M\right|\!A_1 \arrow[r, hookrightarrow] & \dotsc \arrow[r, hookrightarrow] & \left.M\right|\!A_{n-1}  \arrow[r, hookrightarrow] & M
\end{tikzcd}
\end{center}
as $A_1 \subset \dotsc \subset A_{n - 1} \subset M$ or $A_1 \subset \dotsc \subset A_{n - 1} \subset \widetilde{E}(M)$ and denote its isomorphism class by $[M;A_1,\dotsc,A_{n-1}]$. Moreover, we call a flag proper if $A_1$ is non-empty and all other inclusions are proper inclusions. 

\hfill

Now that we have this model for $S_{\bullet}(\mathcal{C})$, let us define the K-theory space. For this, we consider the simplicial space $BS_{\bullet}(\mathcal{C})$ obtained by applying the classifying space functor.
\begin{de}
The (delooped) algebraic K-theory space $K(\mathcal{C})$ of a proto-exact category $\mathcal{C}$ is the geometric realization of the simplicial space $BS_{\bullet}(\mathcal{C})$, and its shifted homotopy groups $\pi_{n + 1}(K(\mathcal{C}))$ are the algebraic K-theory groups $K_n(\mathcal{C})$.
\end{de}
Technically, this should of course not be called the algebraic K-theory space, as this is a delooping of it, but for the sake of simplicity, we will still refer to it as such. As for other types of categories, this is not only a space, but actually carries additional structure.
\begin{lem}\label{Hspace}
The space $K(\Mat_*)$ is the infinite loop space of a spectrum $\Sigma\!\operatorname{K}(\Mat_*)$. In particular, it becomes a connected $H$-space on the level of the homotopy category.
\end{lem}
\begin{proof}
Since $\Mat_*$ has finite coproducts, this can be shown exactly as for exact categories, where it is a standard fact. We refer the reader to \cite{Bar13} for a modern treatment of this fact for the $Q$-construction, and to the proof of Theorem 3.5.3 of \cite{hekking} for a write-up of the equivalence of $Q$- and $S_{\bullet}$-constructions in the setting of proto-exact categories.
\end{proof}
The more combinatorially minded readers, who might not necessarily be familiar with the notions of the above lemma, need not worry. This lemma will be used exactly once, at the end of section $4$, and all we need from it is that the fundamental group of the delooped $K$-theory space acts trivially on the second homotopy group, which is true for these so-called $H$-spaces.

\hfill

Finally, let us note that the $K$-theory and Grothendieck-Witt theory of proto-exact categories was also studied by Eberhardt, Lorscheid and Young in \cite{zbMATH07517045}.

\section{Computing the K-theory of Matroids}
In Theorem 6.3 of \cite{zbMATH07242437}, the authors compute $K_0(\Mat_*)$, using a description of $K_0$ for proto-exact categories with finite coproducts as the free abelian group of isomorphism classes of matroids, up to ``making short exact sequences split''. This is the usual definition of the Grothendieck group. Their result can be stated as follows:
\begin{thm}[\cite{zbMATH07242437}]
The map $(\rk, \lvert \widetilde{E}(-) \rvert - \rk) \colon K_0(\Mat_*) \to \Z^2$ is an isomorphism.
\end{thm}
Moreover, using the free-forgetful adjunction between $\Mat_*$ and $\FinSet_*$ from \cite{zbMATH06859742}, the authors establish the following result as Theorem 6.4 of loc.\ cit.\:
\begin{thm}[\cite{zbMATH07242437}]
There are injective group homomorphisms $\pi_n(\mathbb{S}) \hookrightarrow K_n(\Mat_*)$ for all non-negative integers $n$.
\end{thm}
Putting these two theorems together, we see that at least for $n = 0$, the K-theory of matroids is not determined by the homotopy of the sphere spectrum. But this does raise the question of what happens in other degrees, as Jun, Sistko- and Wright ask in Remark 5.6 of \cite{zbMATH08208138}. Our goal for this section is to give a full answer in degree $1$ and, based on that, give our prediction about the remaining cases of $n \geq 2$.

\hfill

Throughout this section, all matroids are pointed, all strong maps preserve the base point (including automorphisms of matroids), and we let $X = K(\Mat_*)$ denote the (delooped) K-theory space of pointed matroids. As seen in the previous section, it is given as the geometric realization $\lvert BS_{\bullet} \!\Mat_* \rvert$, and its homotopy groups $\pi_{n+1}(X) = K_n(\Mat_*)$ are the algebraic K-theory groups of pointed matroids.

\hfill

The main tool we will be using is the homological geometric realization spectral sequence
$$E^1_{p,q} = H_q(BS_p\!\Mat_*) \Rightarrow H_{p+q}(X).$$
For this entire section, unless we explicitly state otherwise, whenever we write $E_{p,q}^r$ for some integers $p$, $q$ and $r$, we mean the terms of this spectral sequence. Here the differentials $d_{p,q}^1$ are given as the maps induced by the alternating sum over the face maps. For a detailed exposition of this spectral sequence, we refer the reader to Section 12.6.1 of the book project \cite{Mal26}.

The basic idea is to obtain enough information about $H_2(X)$ via the above spectral sequence and then use the Hurewicz map to compare with $\pi_2(X) = K_1(\Mat_*)$.

\hfill

The first group we wish to compute is $E_{1,1}^2$. To have a better understanding of it, let us start by giving a presentation in terms of automorphisms. 
\begin{prop}\label{E11 explicit}
The group $E_{1,1}^2$ is isomorphic to the quotient group
$$\Bigl( \bigoplus_{[M]} \Aut(M)^{\ab} \Bigr) / \langle [M,\alpha] - [\!\left.M\right|\!A,\left.\alpha\right|_A] - [M / A, \overline{\alpha}] \mid M \in \Mat_*, A \subset \widetilde{E}(M), \alpha \in \Aut(M;A) \rangle$$
where $[M]$ runs over all isomorphism classes of matroids $M$.
\end{prop}
\begin{proof}
First note that since $BS_0\!\Mat_*$ is a point, we have $E_{0,1}^1 = 0$ and thus $E_{1,1}^2 = E_{1,1}^1/\im(d_{2,1}^1)$. Thus it suffices to determine those two groups making up this quotient. Lemma \ref{Flag groupoid description} tells us that the first one is
$$H_1(BS_1\!\Mat_*) \cong H_1\Biggl( \coprod_{[M]} B\!\Aut(M) \Biggr) \cong \bigoplus_{[M]} H_1(B\!\Aut(M)) \cong \bigoplus_{[M]} \pi_1(B\!\Aut(M))^{\ab} \cong \bigoplus_{[M]} \Aut(M)^{\ab}$$
where $[M]$ runs over all isomorphism classes of matroids $M$. By the same reasoning, we obtain
$$E_{2,1}^1 \cong \bigoplus_{[A \subset M]} \Aut(M;A)^{\ab}$$
where $[A \subset M]$ runs over all isomorphism classes of flags of length $1$. By another application of Lemma \ref{Flag groupoid description}, the three face maps send a flag $A \subset M$ to the matroids
$$d_0(M;A) = M/A, \hspace{5pt} d_1(M;A) = M, \hspace{5pt} d_2(M;A) = \left.M\right|\!A.$$
On the level of automorphisms, these induce the group homomorphisms
\begin{enumerate}
\item[] $\Aut(M;A) \to \Aut(M/A)$, $\alpha \mapsto \overline{\alpha}$,
\item[] $\Aut(M;A) \to \Aut(M)$, $\alpha \mapsto \alpha$,
\item[] $\Aut(M;A) \to \Aut(\!\left.M\right|\!A)$, $\alpha \mapsto \left.\alpha\right|\!A$
\end{enumerate}
where $\overline{\alpha}$ acts on $E(M/A) = E(M) \setminus A$ by the same underlying permutation as $\alpha$. In particular, the differential $d_{2,1}^1$ is given by
$$d_{2,1}^1([(M;A),\alpha]) = [M / A, \overline{\alpha}] - [M,\alpha] + [\!\left.M\right|\!A,\left.\alpha\right|_A]$$
for an element $\alpha \in \Aut(M;A)$. This proves the claim.
\end{proof}
This group might look familiar to the reader. For an exact category $\mathcal{C}$, this group is also known as Bass' $K_1^{\det}(\mathcal{C})$. It comes equipped with a natural map $K_1^{\det}(\mathcal{C}) \to K_1(\mathcal{C})$, but this map need not be an isomorphism. For more on this group we refer to \cite{zbMATH03358618}, \cite{zbMATH03278303}, \cite{zbMATH03447122} or to the introduction of \cite{zbMATH01156512}.

\hfill

To compute $E_{1,1}^2$ explicitly, we will make use of the following lemma.

\begin{lem}\label{E11 sgn}
There is a well-defined group homomorphism $E_{1,1}^2 \to \{ \pm 1 \}$ mapping the class of an automorphism $[M, \alpha]$ to the sign of the underlying permutation of $\alpha$.
\end{lem}
\begin{proof}
Let $M$ be a matroid with ground set $E(M)$ of cardinality $n$. As $\Aut(M)$ is a subgroup of $S_n$, we get a sign map $\sgn \colon \Aut(M) \to \{ \pm 1 \}$ by restricting the one on $S_n$. Now we factorize it through the abelianization $\Aut(M)^{\ab}$ and obtain a homomorphism $\bigoplus_{[M]} \Aut(M)^{\ab} \to \{ \pm 1 \}$ by taking the direct sum of those maps, which by abuse of notation we still denote by $\sgn$. Finally, we check that this sign map factorizes through the quotient with respect to the image of $d_{2,1}^1$. For this let $A \subset \widetilde{E}(M)$ and let $\alpha \in \Aut(M)$ with $\alpha(A) = A$. Then $\alpha$ also preserves $\widetilde{E}(M) \setminus A$, which gives rise to $\left.\alpha\right|_A$ and $\overline{\alpha}$. Now we factor each of these permutations into transpositions, say $\left.\alpha\right|_A = \tau_1 \cdot \dotsc \cdot \tau_r$ and $\overline{\alpha} = \tau'_1 \cdot \dotsc \cdot \tau'_s$, and extend them to $\widetilde{E}(M)$ by letting them act as the identity outside of their respective domains. This gives a factorization of $\alpha$ into $r + s$ transpositions and results in 
$$\sgn([M,\alpha]) = (-1)^{r + s} = (-1)^r(-1)^s = \sgn([\!\left.M\right|\!A,\left.\alpha\right|_A])\sgn([M / A, \overline{\alpha}]) = \sgn([\!\left.M\right|\!A,\left.\alpha\right|_A] + [M / A, \overline{\alpha}]).$$
Therefore, the sign map respects the relations from $\im(d_{2,1}^1)$ and we obtain the desired sign map on $E_{1,1}^2$.
\end{proof}
\begin{thm}\label{E11}
We have $E^2_{1,1} \cong \mathbb{Z}/2\mathbb{Z}$.
\end{thm}
\begin{proof}
By Proposition \ref{E11 explicit}, the group $E_{1,1}^2$ agrees with the quotient group
$$\Bigl( \bigoplus_{[M]} \Aut(M)^{\ab} \Bigr) / \langle [M,\alpha] - [\left.M\right|\!A,\left.\alpha\right|_A] - [M / A, \overline{\alpha}] \mid M \in \Mat_*, A \subset \widetilde{E}(M), \alpha \in \Aut(M;A) \rangle.$$
To compute it, we show that any automorphism $\alpha$ of a matroid $M$ may be exchanged for the same underlying permutation on the free matroid on $\widetilde{E}(M)$. For this suppose $M$ is not a free matroid, i.e., $\rk(M) < \lvert \widetilde{E}(M) \rvert$. We consider the free coextension $M \times t$ of $M$ by a new element $t$ and use Lemma \ref{auto ext} to extend $\alpha$ to an automorphism $\beta$ of $M \times t$ fixing $t$. By Lemma \ref{basic contr comp}, this yields the relations
$$[M \times t, \beta] = [\left.(M \times t)\right|\!\{ t \}, \left.\beta\right|_{\{ t \}}] + [(M \times t) / \{ t \}, \overline{\beta}] = [U_{1,1}, \id] + [M, \alpha] = [M, \alpha]$$
and
\begin{align*}
[M \times t, \beta] & = [\left.(M \times t)\right|\!\widetilde{E}(M), \left.\beta\right|_{\widetilde{E}(M)}] + [(M \times t) / \widetilde{E}(M), \overline{\beta}] \\ & = [\left.(M \times t)\right|\!\widetilde{E}(M), \left.\beta\right|_{\widetilde{E}(M)}] + [U_{0,1}, \id] \\ & = [\left.(M \times t)\right|\!\widetilde{E}(M), \left.\beta\right|_{\widetilde{E}(M)}],
\end{align*}
from which we obtain $[M,\alpha] = [\left.(M \times t)\right|\!\widetilde{E}(M), \left.\beta\right|_{\widetilde{E}(M)}]$. Since those two matroids have the same ground set and we have $\rk(\!\left.(M \times t)\right|\!\widetilde{E}(M)) = \rk(M) + 1$, we can iterate this process until $[\alpha]$ agrees with the class of its underlying permutation on the free matroid on $\widetilde{E}(M)$. Now we consider a permutation $\sigma \in S_n \cong \Aut(U_{n,n})$. Since $S_n$ is generated by transpositions, we can write $\sigma = \tau_1 \cdot \dotsc \cdot \tau_r$ for some transpositions $\tau_1,\dotsc, \tau_r \in S_n$ and obtain
$$[U_{n,n},\sigma] = [U_{n,n},\tau_1 \cdot \dotsc \cdot \tau_r] = [U_{n,n}, \tau_1] + \dotsc + [U_{n,n}, \tau_r].$$
Therefore, it suffices to understand the class of transpositions. If $(i \,\, j) \in S_n$ is any transposition, the defining relations of $E_{1,1}^2$ give us
$$[U_{n,n},(i \,\, j)] = [\left.U_{n,n}\right|\!\{ i, j \}, \left.(i \,\, j)\right|_{\{ i, j \}}] + [U_{n,n} / \{ i, j \}, \overline{(i \,\, j)}] = [U_{2,2},(1 \,\, 2)] + [U_{n-2,n-2},\id] = [U_{2,2},(1 \,\, 2)].$$
So all transpositions define the same class in $E_{1,1}^2$. In particular, by what we have shown so far, the entire group $E_{1,1}^2$ is generated by $[U_{2,2},(1 \,\, 2)]$. Since
$$2[U_{2,2},(1 \,\, 2)] = [U_{2,2},(1 \,\, 2)^2] = [U_{2,2},\id] = 0,$$
we thus have $E_{1,1}^2 = 0$ or $E_{1,1}^2 \cong \mathbb{Z}/2\mathbb{Z}$. The former cannot be the case, as we can detect the non-triviality of the transposition via the sign homomorphism from Lemma \ref{E11 sgn}. Thus we have shown that $E_{1,1}^2 \cong \mathbb{Z}/2\mathbb{Z}$.
\end{proof}
What the above proof really does is show that we can increase the rank of any matroid until it is full without changing the associated class. Afterwards, we only need to work out the remaining combinatorics on the level of finite sets and their permutations, resulting in the desired computation.

\hfill

Next we will deal with the group $E_{2,0}^2$. By definition, we have
$$E_{1,0}^1 \cong \bigoplus_{[M]} \Z[M], \hspace{5pt} E_{2,0}^1 \cong \bigoplus_{[A \subset M]} \Z[M;A] \hspace{5pt}\text{ and }\hspace{5pt} E_{3,0}^1 \cong \bigoplus_{[A \subset B \subset M]} \Z[M;A,B]$$
where those direct sums run over all isomorphism classes of flags of lengths $0$, $1$ or $2$. As in the proof of Proposition \ref{E11 explicit}, the differentials can be computed as
$$d_{2,0}^1([M;A]) = [M/A] - [M] + [\!\left.M\right|\!A] \hspace{5pt}\text{ and }\hspace{5pt} d_{3,0}^1([M;A,B]) = [M/A; B \setminus A] - [M; B] + [M; A] - [\!\left.M\right|\!B; A]$$
by using Lemma \ref{Flag groupoid description}, which results in the following presentation.
\begin{prop}\label{E20 explicit}
The group $E_{2,0}^2  = \ker(d_{2,0}^1)/\im(d_{3,0}^1)$ is isomorphic to the quotient of the subgroup
$$\Biggl\{ \sum_{[A \subset M]}k_{[M;A]}[M;A] \,\Bigg| \sum_{[A \subset M]}k_{[M;A]}([M / A] - [M] + [\!\left.M\right|\!A]) = 0 \Biggr\} \subset \bigoplus_{[A \subset M]}\Z[M;A]$$
by the group generated by
$$[M/A;B \setminus A] - [M;B] + [M;A] - [\!\left.M\right|\!B;A]$$
for all matroids $M$ and all flags $A \subset B \subset M$ of length $2$.
\end{prop}
Actually, this presentation can be simplified a bit, which we will make use of:
\begin{lem}
In the above presentation, we may restrict to proper flags both for the generators and for the relations.
\end{lem}
\begin{proof}
Let $M$ be a matroid. The two non-proper generators associated with $M$ are $[M; \emptyset]$ and $[M;\widetilde{E}(M)]$. Since
$$d_{3,0}^1([M;\emptyset,\emptyset]) = [M/\emptyset;\emptyset] - [M;\emptyset] + [M;\emptyset] - [\!\left.M\right|\!\emptyset;\emptyset] = [M;\emptyset] - [U_{0,0};\emptyset]$$
and
\begin{align*}
d_{3,0}^1([M;\widetilde{E}(M),\widetilde{E}(M)]) & = [M/\widetilde{E}(M);\emptyset] - [M;\widetilde{E}(M)] + [M;\widetilde{E}(M)] - [\!\left.M\right|\!\widetilde{E}(M);\widetilde{E}(M)] \\ & = [U_{0,0};\emptyset] - [M;\widetilde{E}(M)],
\end{align*}
these are identified with $[U_{0,0};\emptyset]$ for any matroid $M$. In particular, any cycle $x \in \ker(d_{2,0}^1)$ satisfies
$$x \equiv x' + k[U_{0,0};\emptyset] \hspace{-5pt}\mod \im(d_{3,0}^1)$$
for a $\Z$-linear combination $x'$ of classes of proper flags and some integer $k$. As $\im(d_{3,0}^1) \subset \ker(d_{2,0}^1)$, the right-hand side is a cycle. Hence
$$0 = d_{2,0}^1(x' + k[U_{0,0};\emptyset]) = d_{2,0}^1(x') + k d_{2,0}^1([U_{0,0};\emptyset]) = d_{2,0}^1(x') + k[U_{0,0}].$$
But for a proper flag $A \subset M$, none of the terms $[M/A]$, $[M]$ or $[\!\left.M\right|\!A]$ can be $[U_{0,0}]$ as the respective reduced ground sets are all non-empty. This forces $k = 0$, so that $x'$ is a cycle made out of classes of proper flags representing the same class as $x$.

For the three kinds of non-proper flags of length $2$ we have
$$d_{3,0}^1([M;\emptyset,B]) = [M/\emptyset;B \setminus \emptyset] - [M; B] + [M;\emptyset] - [\!\left.M\right|\!B; \emptyset] = [M;\emptyset] - [\!\left.M\right|\!B;\emptyset],$$
$$d_{3,0}^1([M;A,A]) = [M/A;A \setminus A] - [M;A] + [M;A] - [\!\left.M\right|\!A; A] = [M/A; \emptyset] - [\!\left.M\right|\!A;A],$$
\begin{align*}
d_{3,0}^1([M;A,\widetilde{E}(M)]) & = [M/A; \widetilde{E}(M) \setminus A] - [M; \widetilde{E}(M)] + [M;A] - [\!\left.M\right|\!\widetilde{E}(M);A] \\ & = [M/A;\widetilde{E}(M) \setminus A] - [M;\widetilde{E}(M)].
\end{align*}
So the images under $d_{3,0}^1$ of classes of non-proper flags only make use of non-proper flags, and similarly, images under $d_{3,0}^1$ of classes of proper flags only involve proper flags. If a cycle $x \in \ker(d_{2,0}^1)$ built from proper flags lies in the image of $d_{3,0}^1$, we write it as
$$x = d_{3,0}^1(y) + d_{3,0}^1(y')$$
where $y$ is made up of classes of proper flags and $y'$ is made up of classes of non-proper flags. By the above computations, so are the respective images and we obtain $d_{3,0}^1(y') = 0$ by comparing coefficients. In particular, all relations between classes built from proper flags arise from images of classes of proper flags, finishing the proof.
\end{proof}
We will show that the group $E_{2,0}^2$ is cyclic of infinite order, i.e., that it is isomorphic to $\Z$. For this, let us already define the element that will turn out to be a generator of $E_{2,0}^2$. It is the element
$$x_0 = [U_{0,1} \oplus U_{1,1}; \{ f \}] - [U_{0,1} \oplus U_{1,1}; \{ e \}] \in E_{2,0}^1$$
where $\{ e,f \} = \{ e \} \copro \, \{ f \}$ denotes the reduced ground set of the matroid $U_{0,1} \oplus U_{1,1}$. Here $e$ is a loop, that is, it is not contained in any independent set, while $\{ f \}$ is an independent set. We quickly verify that $x_0$ indeed gives rise to an element in $E_{2,0}^2$.
\begin{lem}
The element $x_0$ defines an element in $E_{2,0}^2$. 
\end{lem}
\begin{proof}
We apply the differential and obtain
\begin{align*}
d_{2,0}^1(x_0) & = d_{2,0}^1([U_{0,1} \oplus U_{1,1}; \{ f \}] - [U_{0,1} \oplus U_{1,1}; \{ e \}]) \\
& = [(U_{0,1} \oplus U_{1,1})/\{ f \}] - [U_{0,1} \oplus U_{1,1}] + [\!\left.(U_{0,1} \oplus U_{1,1})\right|\!\{ f \}] \\ & - [(U_{0,1} \oplus U_{1,1})/\{ e \}] + [U_{0,1} \oplus U_{1,1}] - [\!\left.(U_{0,1} \oplus U_{1,1})\right|\!\{ e \}] \\ & = [U_{0,1}] - [U_{0,1} \oplus U_{1,1}] + [U_{1,1}] - [U_{1,1}] + [U_{0,1} \oplus U_{1,1}] - [U_{0,1}] \\ & = 0,
\end{align*}
so that $x_0$ yields a class $\overline{x_0} \in \ker(d_{2,0}^1)/\im(d_{3,0}^1) = E_{2,0}^2$.
\end{proof}
So we have a special element to work with. Next we need to figure out how to relate it to elements $\overline{[M;A]} \in E_{2,0}^1/\im(d_{3,0}^1)$. For this we observe that we can reduce the size of $A$ using the relations stemming from the differential $d_{3,0}^1$. Picking an element $e_1 \in A$, we get
$$[M;A] \equiv [M;\{ e_1 \}] + [M/\{ e_1 \}; A \setminus \{ e_1 \}] - [\!\left.M\right|\!A; \{ e_1 \}]  \hspace{-5pt}\mod \im(d_{3,0}^1).$$
Here the first and the last terms make use of singletons and are hence a lot simpler, and the term in the middle also uses a subset of smaller size. Therefore, we can iterate this for the middle term, resulting in rather simple terms after enough iterations. Thus it makes sense to group all the singleton terms arising in this process, which we will do now. For any isomorphism class of matroids $[M]$ we choose a representative $M$ together with an ordering $\widetilde{E}(M) = \{ e_1, \dotsc, e_n \}$ of its reduced ground set. This allows us to define a group homomorphism
$$z \colon E_{1,0}^1 \cong \bigoplus_{[M]} \Z[M] \to \bigoplus_{[A \subset M]} \Z[M;A] \cong E_{2,0}^1$$
by mapping a basis element $[M]$ to $z(M;e_1,\dotsc,e_n) = \sum_{i = 1}^{n-1} [M/\{ e_1,\dotsc,e_{i-1} \}; \{ e_i \}]$ if $\lvert \widetilde{E}(M) \rvert \geq 2$ and to $0$ if $\lvert \widetilde{E}(M) \rvert \leq 1$. As we chose an ordering of the $e_i$, we now investigate how much the map $z$ depends on this ordering.
\begin{prop}
For any permutation $\sigma \in S_n$, we have 
$$z(M;e_1,\dotsc,e_n) \equiv z(M;e_{\sigma(1)},\dotsc,e_{\sigma(n)}) \hspace{-5pt}\mod \im(d_{3,0}^1) + \Z \cdot \, x_0.$$
\end{prop}

\begin{proof}
If $\widetilde{E}(M)$ consists of two elements $e_1$ and $e_2$, the defining formula of $z$ directly yields
$$z(M;e_1,e_2) - z(M;e_2,e_1) = [M;\{e_1\}] - [M;\{e_2\}].$$
We will now show that this is true more generally when swapping only the first two elements of $\widetilde{E}(M)$ for any matroid $M$, at least when working mod $\im(d_{3,0}^1)$. For this we let $M$ be a matroid with $\lvert \widetilde{E}(M) \rvert \geq 3$ and swap $e_1$ with $e_2$. Then we have
$$z(M;e_2,e_1,e_3,\dotsc,e_n) = [M;\{ e_2\}] + [M/\{ e_2 \}; \{ e_1\}] + \sum_{i = 3}^{n-1} [M/\{ e_1,\dotsc,e_{i-1} \}; \{ e_i \}]$$
since the swap does not matter once both $e_1$ and $e_2$ have been contracted. In particular, we obtain
$$z(M;e_1,\dotsc,e_n) - z(M;e_2,e_1,e_3,\dotsc,e_n) = [M;\{ e_1 \}] + [M/\{ e_1 \}; \{ e_2\}] - [M;\{ e_2 \}] - [M/\{ e_2 \}; \{ e_1\}].$$
Using the explicit formula for the differential $d_{3,0}^1$, we see that in the quotient $E_{2,0}^1/\im(d_{3,0}^1)$, the first two summands agree with 
$$[M;\{ e_1,e_2\}] + [\!\left.M\right|\!\{ e_1,e_2\}; \{ e_1 \}],$$
and the latter two summands agree with
$$-([M;\{ e_1,e_2\}] + [\!\left.M\right|\!\{ e_1,e_2\}; \{ e_2 \}]).$$
Thus we have
$$z(M;e_1,\dotsc,e_n) - z(M;e_2,e_1,e_3,\dotsc,e_n) \equiv [\!\left.M\right|\!\{ e_1,e_2\}; \{ e_1 \}] - [\!\left.M\right|\!\{ e_1,e_2\}; \{ e_2 \}] \hspace{-5pt}\mod \im(d_{3,0}^1),$$
as claimed above. Exchanging two adjacent $e_i$ at a later position leaves earlier and even later summands unchanged, because the respective contracted sets agree. Only the exchanged elements contribute to the difference, computed as above. Since $\lvert\widetilde{E}((\!\left.M / \{ e_1,\dotsc, e_{i-1} \}) \right|\!\{ e_i,e_{i+1}\}) \rvert = 2$, matroids of this form must be isomorphic to one of the matroids
$$U_{0,2}, \hspace{5pt} U_{1,2}, \hspace{5pt} U_{2,2} \hspace{5pt} \text{ or } \hspace{5pt} U_{0,1} \oplus U_{1,1}.$$
For the first three, swapping the two elements is an automorphism, which hence results in no difference. For $U_{0,1} \oplus U_{1,1}$, the difference is by definition given by $\pm x_0$. As any permutation $\sigma \in S_n$ can be written as a product of adjacent transpositions, this yields
$$z(M;e_1,\dotsc,e_n) \equiv z(M;e_{\sigma(1)},\dotsc,e_{\sigma(n)}) \hspace{-5pt}\mod \im(d_{3,0}^1) + \Z \cdot \,x_0,$$
which is what we wanted to show.
\end{proof}
Due to this proposition, we will allow ourselves to write $z(M)$ instead of $z(M;e_1,\dotsc,e_n)$ as long as we work mod $\im(d_{3,0}^1) + \Z \cdot \, x_0$. As the motivation for the map $z$ hinted at, these kinds of elements allow us to express the class of a flag $[M;A]$ in terms of $M$, its restriction to $A$ and its contraction with respect to $A$ as follows.
\begin{prop}\label{class of flag in terms of z}
For any matroid $M$ and any non-empty proper subset $A \subset \widetilde{E}(M)$, we have
$$[M;A] \equiv z(M) - z(\!\left.M\right|\!A) - z(M/A) \hspace{-5pt}\mod \im(d_{3,0}^1) + \Z \cdot \, x_0.$$
\end{prop}
\begin{proof}
For a proper flag $A \subset M$, we choose an ordering $\widetilde{E}(M) = \{ e_1,\dotsc,e_n \}$ so that the elements of the subset $A = \{ e_1,\dotsc, e_m \}$ come first, where $1 \leq m \leq n - 1$. Now we claim that
$$[M;A] \equiv z(M; e_1,\dotsc, e_n) - z(\!\left.M\right|\!A; e_1,\dotsc,e_m) - z(M/A;e_{m+1},\dotsc,e_n) \hspace{-5pt}\mod \im(d_{3,0}^1).$$
We consider the proper flag 
$$\{ e_i \} \subset A \setminus \{ e_1, \dotsc, e_{i-1} \} \subset M / \{ e_1, \dotsc, e_{i-1} \}$$
for all $1 \leq i \leq m-1$. Its image under $d_{3,0}^1$ is given by the sum of
$$[M / \{ e_1, \dotsc, e_i \}; A \setminus \{ e_1, \dotsc, e_i \}] - [M / \{ e_1, \dotsc, e_{i-1} \}; A \setminus \{ e_1, \dotsc, e_{i - 1} \}]$$
and
$$[M / \{ e_1, \dotsc, e_{i-1} \}; \{ e_i \}] - [\!\left.(M\right|\!A) / \{ e_1, \dotsc, e_{i - 1} \}; \{ e_i \}]$$
where we made use of Lemma \ref{contraction and restrictrion commute}. Summing the first two summands over all $i$ gives
$$[M / \{ e_1,\dotsc, e_{m-1} \}; \{ e_m \}] - [M;A],$$
so that we obtain
$$\sum_{i = 1}^m [M/\{ e_1,\dotsc, e_{i-1} \}; \{ e_i \}] - [M;A] = z(M; e_1, \dotsc, e_n) - z(M/A; e_{m + 1}, \dotsc, e_n) - [M;A]$$
after adding the sum $\sum_{i = 1}^{m-1} [M / \{ e_1, \dotsc, e_{i-1} \}; \{ e_i \}]$. Finally, we have 
$$\sum_{i = 1}^{m-1}[\!\left.(M\right|\!A) / \{ e_1, \dotsc, e_{i - 1} \}; \{ e_i \}] = z(\!\left.M\right|\!A; e_1,\dotsc,e_m),$$ which yields the claimed formula. 
\end{proof}
\begin{thm}\label{E20}
We have $E^2_{2,0} \cong \Z$.
\end{thm}
\begin{proof}
As we have mentioned already, we claim that $\overline{x_0}$ generates the group $E_{2,0}^2$. To verify this, we let $\sum_{i = 1}^r k_i [M_i;A_i] \in \ker(d_{2,0}^1)$ be a cycle of proper flags. Then the previous proposition yields
\begin{align*}
\sum_{i = 1}^r k_i [M_i;A_i] & \equiv \sum_{i = 1}^r k_i (z(M_i) - z(\!\left.M_i\right|\!A_i) - z(M_i/A_i)) \\ & = -z\biggl( \sum_{i = 1}^r k_i ([M_i/A_i] - [M_i] + [\!\left.M_i\right|\!A_i]) \biggr) \\ & = -z\biggl( \sum_{i = 1}^r k_i d_{2,0}^1([M_i;A_i]) \biggr) \\ & = -z(0) \\ & = 0
\end{align*}
mod $\im(d_{3,0}^1) + \Z \cdot \,x_0$. In particular, we have $\sum_{i = 1}^r k_i [M_i;A_i] \equiv k \cdot x_0$ mod $\im(d_{3,0}^1)$ for a suitable $k \in \Z$. In other words, $\overline{x_0}$ generates the group $E_{2,0}^2$. It remains to show that $\overline{x_0}$ has infinite order. For this we define a group homomorphism 
$$\bigoplus_{[A \subset M]} \mathbb{Z} \cdot [M;A] \to \Z$$
by mapping a basis element $[M;A]$ to $\rk_M(A) \cdot \lvert \widetilde{E}(M) \setminus A \rvert$, and denote its restriction to $\ker(d_{2,0}^1)$ by $\varphi$. Since for any generator $[M; A,B] \in E_{3,0}^1$, we have
\begin{align*}
\varphi(d_{3,0}^1([M; A,B])) & = \varphi([M/A; B \setminus A] - [M; B] + [M; A] - [\left.M\right|\!B; A]) \\ & = \rk_{M / A}(B \setminus A) \cdot \lvert \widetilde{E}(M / A) \setminus (B \setminus A) \rvert - \rk_{M}(B) \cdot \lvert \widetilde{E}(M) \setminus B \rvert
\\ & + \rk_{M}(A) \cdot \lvert \widetilde{E}(M) \setminus A \rvert - \rk_{\!\left.M\right|B}(A) \cdot \lvert \widetilde{E}(\!\left.M\right|\!B) \setminus A \rvert \\
& = (\rk_M(B)-\rk_M(A)) \cdot \lvert \widetilde{E}(M) \setminus B \rvert - \rk_{M}(B) \cdot \lvert \widetilde{E}(M) \setminus B \rvert
\\ & + \rk_M(A) \cdot \lvert \widetilde{E}(M) \setminus A \rvert - \rk_M(A) \cdot \lvert B \setminus A \rvert \\
& = \rk_M(A)(\lvert \widetilde{E}(M) \setminus A \rvert - \lvert B \setminus A \rvert - \lvert \widetilde{E}(M) \setminus B \rvert) \\ 
& = 0,
\end{align*}
we get an induced homomorphism $\overline{\varphi} \colon E_{2,0}^2 \to \Z$. It maps $\overline{x_0}$ to
$$\rk_{U_{0,1} \oplus \, U_{1,1}}(\{ f \}) \cdot \lvert \widetilde{E}({U_{0,1} \oplus U_{1,1}}) \setminus \{ f \} \rvert - \rk_{U_{0,1} \oplus \, U_{1,1}}(\{ e \}) \cdot \lvert \widetilde{E}({U_{0,1} \oplus U_{1,1}}) \setminus \{ e \} \rvert = 1 \cdot 1 - 0 \cdot 1 = 1,$$
so that the order of $\overline{x_0}$ is infinite. This finishes the proof.
\end{proof}
\begin{rem}
Although we do not want to get into detail, let us quickly note how the homomorphism $\varphi$ used in the proof of Theorem \ref{E20} arises. There is a natural homomorphism which maps a flag to the class of its successive quotients in 
$$K_0(\Mat_*) \times K_0(\Mat_*) \subset \Z[K_0(\Mat_*) \times K_0(\Mat_*)].$$
Evaluating the cup product of the rank and reduced cardinality maps $K_0(\Mat_*) \to \Z$ viewed as representatives of elements in $H^1(K_0(\Mat_*))$ on these classes then yields $\varphi$ up to a sign; see the explicit formula at the top of page 110 of \cite{zbMATH03935317} for $p = 1 = q$.
\end{rem}
Based on the previous results, we have almost all the needed ingredients to compute the group $K_1(\Mat_*)$. To complete this computation, we will compare two short exact sequences, the first one being:
\begin{cor}\label{ses quotient Zmod2}
There is a short exact sequence of abelian groups of the form
\begin{center}
\begin{tikzcd}
0 \arrow[r] & \coker(d^2_{3,0} \colon E^2_{3,0} \to \mathbb{Z}/2\mathbb{Z}) \arrow[r] & H_2(X) \arrow[r] & \Z \arrow[r] & 0.
\end{tikzcd}
\end{center}
In particular, we have a splitting $H_2(X) \cong \coker(d^2_{3,0} \colon E^2_{3,0} \to \mathbb{Z}/2\mathbb{Z}) \oplus \Z$.
\end{cor}
\begin{proof}
First note that the line $E_{0,q}^1$ vanishes for positive $q$ since $S_0\!\Mat_*$ is the terminal groupoid. In particular, we have $E^2_{0,2} = 0$ and hence $E^\infty_{0,2} = 0$. Moreover, we have $E^2_{1,1} \cong \mathbb{Z}/2\mathbb{Z}$ and $E^2_{2,0} \cong \Z$ by Theorem \ref{E11} and Theorem \ref{E20}. For degree reasons, no non-trivial differential can enter at the position $(2,0)$ for $r \geq 2$ and the only potentially non-trivial differential leaving this position for $r \geq 2$ is $d^2_{2,0} \colon E^2_{2,0} \to E^2_{0,1}$. As the line $p = 0$ vanishes for positive $q$, the differential $d^2_{2,0}$ is the zero map. Therefore, $E^{\infty}_{2,0} = E^2_{2,0} \cong \Z$. Analogously we obtain 
$$E^\infty_{1,1} = \coker(d^2_{3,0} \colon E^2_{3,0} \to E^2_{1,1}) = \coker(d^2_{3,0} \colon E^2_{3,0} \to \mathbb{Z}/2\mathbb{Z})$$
for degree reasons. This yields the filtration
$$0 = F_{-1}H_2(X) \subset F_0H_2(X) \subset F_1H_2(X) \subset F_2H_2(X) = H_2(X)$$
with subquotients given by:
\begin{enumerate}
\item[] $F_0H_2(X) = F_0H_2(X)/F_{-1}H_2(X) = E^\infty_{0,2} = 0$,
\item[] $F_1H_2(X) = F_1H_2(X)/F_0H_2(X) = E^{\infty}_{1,1} \cong \coker(d^2_{3,0} \colon E^2_{3,0} \to \mathbb{Z}/2\mathbb{Z})$,
\item[] $F_2H_2(X)/F_1H_2(X) = E^\infty_{2,0} \cong \Z$.
\end{enumerate}
Therefore, the short exact sequence
\begin{center}
\begin{tikzcd}
0 \arrow[r] & F_1H_2(X) \arrow[r, hookrightarrow] & F_2H_2(X) \arrow[r, twoheadrightarrow] & F_2H_2(X)/F_1H_2(X) \arrow[r] & 0
\end{tikzcd}
\end{center}
shows the claim.
\end{proof}
In other words, the second homology of the K-theory space is either $\Z$ or $\mathbb{Z}/2\mathbb{Z} \oplus \Z$. As we are interested in the K-groups, that is, in the homotopy of the K-theory space, we now want to compare this homology group with $\pi_2(X)$. 
\begin{prop}\label{ses inclusion Zmod2}
There is a short exact sequence of abelian groups of the form
\begin{center}
\begin{tikzcd}
0 \arrow[r] & \pi_2(X) \arrow[r] & H_2(X) \arrow[r] & \Z \arrow[r] & 0.
\end{tikzcd}
\end{center}
In particular, we have a splitting $H_2(X) \cong \pi_2(X) \oplus \Z$.
\end{prop}
\begin{proof}
From obstruction theory we have the fiber sequence
\begin{center}
\begin{tikzcd}
\widetilde{X} \simeq X_{\geq 2} \arrow[r] & X \arrow[r] & X_{\leq 1} \simeq B\pi_1(X)
\end{tikzcd}
\end{center}
where $\widetilde{X}$ denotes the universal covering of $X$, and we can hence consider the associated homological Serre spectral sequence 
$$E_{p,q}^2 = H_p(B\pi_1(X);H_q(\widetilde{X})) \Rightarrow H_{p+q}(X).$$
To avoid confusion, let us explicitly state that the following $E_{p,q}^r$ terms refer to this spectral sequence. Part of its low-degree exact sequence is given by
\begin{center}
\begin{tikzcd}
H_1(B\pi_1(X); H_1(\widetilde{X})) \arrow[r] & H_2(X)/F_0H_2(X) \arrow[r] & H_2(B\pi_1(X)) \arrow[r] & H_0(B\pi_1(X); H_1(\widetilde{X})), 
\end{tikzcd}
\end{center}
but since $\widetilde{X}$ is simply connected, we have $H_1(\widetilde{X}) = 0$, so the outer two terms vanish. In particular, the map in the middle is an isomorphism.

Now we want to determine $F_0H_2(X) \cong E_{0,2}^\infty$. For this we first note that the line at $q = 1$ vanishes due to $H_1(\widetilde{X})$ being trivial. This immediately yields 
$$E_{0,2}^3 = \ker(d^2_{0,2})/\im(d^2_{2,1}) = E_{0,2}^2 = H_0(B\pi_1(X);H_2(\widetilde{X}))$$
and
$$E_{3,0}^3 = \ker(d^2_{3,0})/\im(d^2_{5,-1}) = H_3(B\pi_1(X)).$$
For degree reasons, no other potentially non-trivial differential enters or leaves the $(0,2)$-position on higher pages. Thus we get
$$F_0H_2(X) = E_{0,2}^\infty = \coker(d^3_{3,0}) = \coker(H_3(B\pi_1(X)) \to H_0(B\pi_1(X);H_2(\widetilde{X}))).$$
Now we splice this into the above exact sequence as follows:
\begin{center}
\begin{tikzcd}
H_3(B\pi_1(X)) \arrow[r] & H_0(B\pi_1(X);H_2(\widetilde{X})) \arrow[r] & H_2(X) \arrow[r] & H_2(B\pi_1(X)) \arrow[r] & 0 
\end{tikzcd}
\end{center}
By Theorem 6.3 of \cite{zbMATH07242437}, we know that $\pi_1(X) \cong \Z^2$ and hence that $B\pi_1(X) \simeq T^2$ is a torus. In particular, $H_3(B\pi_1(X)) = 0$ and $H_2(B\pi_1(X)) \cong \Z$. Next we use that the Hurewicz map is $\pi_1(X)$-equivariant, so that we obtain
$$H_0(B\pi_1(X);H_2(\widetilde{X})) \cong H_2(\widetilde{X})_{\pi_1(X)} \cong \pi_2(\widetilde{X})_{\pi_1(X)} \cong \pi_2(X)_{\pi_1(X)}.$$
Finally, $X$ becomes a connected $H$-space by Lemma \ref{Hspace}, so the action of $\pi_1(X)$ on $\pi_2(X)$ is trivial and we are done.
\end{proof}
With a tiny bit more care, the map $\pi_2(X) \to H_2(X)$ in the above short exact sequence can be seen to be the Hurewicz map, but we do not need this. Now we can put everything together and obtain the desired computation:
\begin{thm}\label{main result}
We have $K_1(\Mat_*) = \pi_2(X) \cong \mathbb{Z}/2\mathbb{Z}$.
\end{thm}
\begin{proof}
The previous proposition together with Corollary \ref{ses quotient Zmod2} gives us an isomorphism
$$\coker(d^2_{3,0} \colon E^2_{3,0} \to \mathbb{Z}/2\mathbb{Z}) \oplus \Z \cong H_2(X) \cong \pi_2(X) \oplus \Z.$$
By Walker's cancellation theorem \cite{zbMATH03119318} (or \cite{zbMATH03117594}), we can cancel the two copies of $\Z$, so $\pi_2(X)$ is a quotient of $\mathbb{Z}/2\mathbb{Z}$. Since $\pi_1(\mathbb{S}) \cong \mathbb{Z}/2\mathbb{Z}$ injects into $\pi_2(X)$ by Theorem 6.4 of \cite{zbMATH07242437}, we hence obtain $\pi_2(X) \cong \mathbb{Z}/2\mathbb{Z}$ as claimed.
\end{proof}
Since one of the groups is finite, we could have also shown directly why the copies of $\Z$ can be cancelled. Indeed, if $f \colon A \oplus \Z \to B \oplus \Z$ is an isomorphism of abelian groups with $A$ finite, then the composition with the projection $B \oplus \Z \to \Z$ must annihilate the torsion group $A \oplus 0$. Since the composite is surjective, it needs to be multiplication by $\pm 1$ on the $\Z$ component, so that $A \oplus 0$ is the kernel of this map. As this kernel also agrees with $f^{-1}(B \oplus 0)$, $f$ restricts to an isomorphism $A \to B$, which is what we wanted to show.

\hfill
 
Now that we have answered the question convering $K_1(\Mat_*)$ in Remark 5.6 of \cite{zbMATH08208138}, let us comment on what we expect from the higher K-theory groups. Based on the computation of $E_{1,1}^2$ (where we showed that we can increase the rank of a matroid without changing its class), we believe that the K-theory of pointed matroids does not really differ much from that of pointed finite sets. More precisely:
\begin{conj}\label{conj}
The map $\operatorname{K}(\Mat_*) \to \HH\!\Z \vee \, \mathbb{S}$ induced by the rank map and by the forgetful map to pointed finite sets is an equivalence of spectra. In particular, we have $K_n(\Mat_*) \cong \pi_n(\mathbb{S})$ for all positive $n$.
\end{conj}
Based on Theorem 6.3 of \cite{zbMATH07242437} and on this article, this conjecture is currently known at the level of homotopy groups in degrees $0$ and $1$.
\bibliographystyle{alpha}
\bibliography{references}

\end{document}